\documentclass[a4paper, 11pt]{amsart}

\usepackage{amsmath,amssymb,amsthm,amscd,enumerate,setspace}
\usepackage[all]{xy}
\usepackage[dvips]{graphicx}

\usepackage{hyperref}

\input xypic
\xyoption{all}

\newcommand{\rar}{\rightarrow}
\newcommand{\lar}{\longrightarrow}

\newcommand{\llar}{-\kern-5pt-\kern-5pt\longrightarrow}
\newcommand{\lllar}{-\kern-5pt-\kern-5pt\llar}

\long\def\symbolfootnote[#1]#2{\begingroup%
\def\thefootnote{\fnsymbol{footnote}}\footnote[#1]{#2}\endgroup}

\newtheorem{Theorem}{Theorem}[section]

\newtheorem{Corollary}[Theorem]{Corollary}
\newtheorem{Proposition}[Theorem]{Proposition}

\newtheorem{Remark}[Theorem]{Remark}
\newtheorem{Example}[Theorem]{Example}

\def\depth{\mbox{\rm depth }}

\def\gr{\mbox{\rm gr}}
\def\reg{\mbox{\rm reg}}
\def\red{\mbox{\rm red}}

\def\e{\mathrm{e}}

\def\m{\mathfrak{m}}

\def\AA{{\mathbf A}}

\def\BB{{\mathbf B}}

\def\DD{{\mathbf D}}

\def\RR{{\mathbf R}}
\def\PP{{\mathbf P}}

\def\g2{{\mathbf g}}

\def\H{{\mathrm H}}

\def\m{{\mathfrak m}}

\begin{document}

\title{Ideals generated by quadrics }

\thanks{AMS 2010 {\em Mathematics Subject Classification}. 
Primary 13A30; Secondary 13F20, 13H10, 13H15. 
The second author is partially supported by a CNPQ Research Fellowship (Brazil) and a CAPES Senior Visiting Reseach Scholarship (Brazil).}

\author{Jooyoun Hong}
\address{ Department of Mathematics \\
Southern Connecticut State University \\
501 Crescent Street, New Haven, CT 06515-1533, U.S.A.}
\email{hongj2@southernct.edu}

\author{Aron Simis}
\address{Departamento de Matem\'atica \\
Universidade Federal de Pernambuco \\
50740-540 Recife, PE, Brazil }
\email{aron@dmat.ufpe.br}

\author{Wolmer V. Vasconcelos}
\address{Department of Mathematics \\ Rutgers University \\
110 Frelinghuysen Rd, Piscataway, NJ 08854-8019, U.S.A.}
\email{vasconce@math.rutgers.edu}


\begin{abstract}
\noindent
Our purpose is to study the cohomological properties of  the Rees algebras of a class of ideals generated by 
quadrics. For {\em all} such ideals $I\subset \RR = K[x,y,z]$ we give the precise value of $\depth  \RR[It]$ 
and decide whether the corresponding rational maps are birational. In the case of dimension $d \geq 3$, when $K=\mathbb{R}$, we give structure theorems for 
all ideals of codimension $d$ minimally generated by ${{d+1}\choose{2}}-1$ quadrics. 
For arbitrary fields $K$, we prove a polarized version.

\end{abstract}

\maketitle

\section{Introduction}

\noindent
Let $K$ be a field, $\RR=K[x_1, \ldots, x_d]$ and $\m=(x_1, \ldots, x_{d})$. We say that an ideal $I$ of $R$ is {\em submaximally generated} if it has codimension $d$ and is minimally generated by $\nu(\m^{2})-1$ quadrics.
Even with these restrictions, 
the cohomological properties of the Rees algebras $\RR[It]$, wrapped around the calculation of their depths, can
be daunting. In two circumstances we give a full picture of these algebras.

\medskip

The first result--Theorem~\ref{5quadrics}--asserts,  for all ideals $I$ of codimension $3$ minimally generated by $5$ quadrics in $K[x,y,z]$,  an interesting uniformity:  the rational mapping defined by $I$ is birational, $\RR[It]$ satisfies the condition $ R_1$ of 
 Serre and $\depth \RR[It] = 1$ or $3$ depending on whether $ I$ is Gorenstein or not.

 We recall that a set of forms $f_1,\ldots,f_m\in \RR$ of fixed degree $n$ defines a {\em birational} map if the field extension
 $$K(f_1,\ldots,f_m) \subset K((x_1, \ldots, x_d)_n)$$
 is an equality, where the right most field is the Veronese field of order $n$.
 A basic result (\cite[Proposition 3.3]{syl2})  states that, for an $\m$-primary ideal $I\subset \RR$, birationality is equivalent to $\RR[It]$ having the condition $ R_1$ of Serre.

\medskip

In all dimensions $d$, if $K = \mathbb{R}$,  two structure results--Theorem~\ref{abc} and Theorem~\ref{abc2}--describe first all the Gorenstein
ideals of codimension $d$ minimally generated by $\nu(\m^{2})-1$ quadrics 
and then how these ideals are assembled together
to describe all others. 
 The corresponding vector spaces of quadrics
are always spanned by binomials and are distributed into finitely many orbits of the natural action of the linear group.
In all characteristics, Theorem~\ref{abc3} gives a polarized version of Theorem~\ref{abc}, with similar consequences.   

\section{Ideals generated by 5 quadrics}

Throughout this section $\RR$ is the ring of polynomial $K[x,y,z]$ over the field $K$. In a few occasions
when dealing with Gorenstein ideals we
assume that char $K\neq 2$. Our purpose is to prove the following assertion
about all ideals of codimension $3$ generated by $5$ quadrics.

\begin{Theorem} \label{5quadrics}
Let $\RR = K[x,y,z]$, 
 and $I$ an ideal of codimension $3$ minimally generated by $5$ quadrics.
\begin{itemize}
\item[{\rm (i)}] If $\mbox{\rm char}\, K \neq 2$ and $I$ is Gorenstein, then $I$ is syzygetic. Conversely, in all characterics,  if $I$ is syzygetic, then $I$ is Gorenstein.
 \item[{\rm (ii)}] If $I$ is syzygetic, then
  $\depth   \gr_{I}(\RR) = 0$.
\item[{\rm (iii)}] If $I$ is not Gorenstein, then
 $\depth   \gr_{I}(\RR) = 2$.
\item[{\rm (iv)}] 
$\RR[It]$ has the condition $R_1$ of Serre, i.e., the set of quadrics
generating $I$ is birational.
\item[{\rm (v)}]  $\red (I)= 2$.
\end{itemize}
\end{Theorem}

\medskip

Let $I$ be an ideal of codimension $3$ generated by quadrics of $\RR=K[x,y,z]$,
 $\m=(x,y,z)$.  As a general reference we shall use \cite{BH} and unexplained terminology
will be defined when needed.

\medskip

We begin with some observations about the two Hilbert functions associated to $I$, i.e., the Hilbert function of $\RR/I$ and that of  the associated graded ring $\gr_{I}(\RR) = \sum_{i\geq 0}I^i/I^{i+1}$ of the ideal $I$. 
In order to determine the Hilbert function of $\RR/I$, we may extend the field $K$ to assume that $K$ is infinite. 
 Let $J$ be a minimal reduction of $I$ generated by $3$ quadrics. The Hilbert function of $\RR/J$
is $(1,3,3,1)$, with the socle in degree $3$. If $I$ is minimally generated by $5$ quadrics, 
 $\lambda(I/J)\geq 2$, that is $5\leq \lambda(\RR/I) \leq 6$. 
  The possible Hilbert functions of $\RR/I$ are $(1,3,1)$ or $(1,3,1,1)$. The second cannot occur since
the socle of $\RR/J$ will map to $0$ in $\RR/I$. Hence $\lambda(\RR/I)=5$.   
If $I$ is minimally generated by $4$ quadrics, a similar analysis shows that the Hilbert
function of $\RR/I$ is $(1,3,2)$. Hence $\lambda(\RR/I)=6$.

\medskip

The other Hilbert function is that of $\gr_{I}(\RR) = \sum_{i\geq 0}I^i/I^{i+1}$ of the ideal $I$. Its two leading coefficients $\e_0(I)$ and $\e_1(I)$ are known to some extent: since
the integral closure of $I$ is $\m^2$,  $\e_0(I) = \e_0(\m^2)=8$ and $\e_1(I)\leq \e_1(\m^2)=4$. 

\medskip

More generally, suppose that $\RR = K[x_1, \ldots, x_d]$ with $\m = (x_1, \ldots, x_d)$ and 
 that   $I$ is a  Gorenstein ideal of codimension $d$ minimally generated by
$\nu(\m^{2})-1$ quadrics. Then the Hilbert function of $\RR/I$ is simply $(1, d, 1)$, because once a 
component has dimension $1$ those that follow also have dimension $1$ or $0$. The symmetry in the case of 
Gorenstein  algebras show that the Hilbert function has to be of the type indicated.

\medskip

The proof of Theorem~\ref{5quadrics} is the assembly  of Propositions \ref{Prop2.1}, \ref{Prop2.2} and \ref{Prop2.3}.
We recall the notion of syzygetic ideals. Given an ideal $I$, there is a canonical mapping from its symmetric square $S_{2}$ onto $I^{2}$:
\[ 0 \rar \delta(I) \rar S_{2} \rar I^{2} \rar 0.\]
An ideal $I$ is called {\em syzygetic} if $\delta(I)=0$. If $I$ is syzygetic, then the minimal number of generators of $I^2$,
 $\nu(I^{2})$, is given by the formula $  {{\nu(I)+1}\choose{2}}  $. The converse however is not true. Another elementary property of these ideals is that $I^2\neq JI$ for any proper subdeal $J\subset I$.  
 
 \medskip
 
We recall that if $1/2\in K$ and $I$ is  Gorenstein of codimension $3$, then $I$ syzygetic
(\cite[Proposition 2.8]{HSV83}). 

\begin{Proposition}\label{Prop2.1}
Let $\RR = K[x,y,z]$, and $I$ an ideal of codimension $3$ generated by $5$ quadrics. If $I$ is
syzygetic, then
\begin{itemize}
\item[{\rm (i)}] $\RR[It]$ has the condition $R_1$ of Serre;
\item[{\rm (ii)}] $\depth   \gr_{I}(\RR) = 0$; 
\item[{\rm (iii)}] $\red (I)= 2$.
\end{itemize}
\end{Proposition}

\begin{proof} Let $\m$ be a maximal ideal containing $I$.
Since $I$ is syzygetic, its square is minimally generated by ${{6}\choose{2}}=15$ quartics and 
therefore we must have $I^2 = \m^4$. 
From the Hilbert function of $R/I$, we have $\m^{3}= \m I$. Hence
\[I^2= \m^{4}= \m \m^{3}= \m^2 I,\] and 
\[ \m^6 = \m^2 I^2 = I^3.\] Similarly we have $\m^{2n} = I^n$ for $n \geq 2$. [This implies, of course, that $I$
and $\m^2$ have the  same Hilbert polynomial.] 
Thus in the embedding of Rees algebras
\[ 0 \rar \RR[It] \lar \RR[\m^2t] \lar C \rar 0,\]
$C$ is a module of finite length.
Since $\RR[\m^2t]$ is normal, $\RR[It] $ has the condition $R_1$ and depth $1$. By the proof of \cite[Proposition 1.1]{Hu0}, we get $\depth   \gr_{I}(\RR) = 0$.

\medskip

The proof of (iii) is clear: If $J$ is a minimal reduction of $I$, it is also a minimal reduction of $\m^2$, an 
ideal of reduction number $1$,
\[ I^3= \m^6 = J\m^4 = JI^2.\]
Since $I^{2}\neq JI$, the reduction number of $I$ is $2$.
\end{proof}

Of course we don't know yet of other classes of syzygetic ideals. This is clarified in the next result.
Let us discuss non-Gorenstein ideals.

\begin{Proposition}\label{Prop2.2} 
Let $R=K[x,y,z]$ and $I$  an ideal of codimension $3$ minimally generated by $5$ quadrics. Suppose that $I$ is not Gorenstein. Then
$I$ is either $ ( x^2, xy, xz, y^2, z^2)$ or $ ( x^2, xy, xz, yz, ay^2 +cz^2)$ up to change of variables. In particular, $I$ is not syzygetic and $\red (I)= 2$.
\end{Proposition}

\begin{proof} Let $\m=(x,y,z)$. 
Since $\lambda(\m^2/I)= 1$, $\m^2$ is contained in the socle of $I$. A larger
socle would contain a linear form $h$, so $h\cdot \m\subset I$. 
We may assume that $K$ is infinite. We change variables and set $x=h$. With $x(x,y,z) \subset I$, $I$ can be written as 
\[ I = (x\cdot \m, q_1, q_2),\]
where $q_1$ and $q_2$ are quadratic forms in $y,z$.

\medskip

 If $q_1 = ay^2+byz+cz^2$ and $q_2 = dy^2+eyz+fz^2$, we cannot have $a=d=0$, but by subtraction we can assume
that one of them is $0$ and the other nonzero. If $q_1 = (by + cz)z$ we can by change of variables assume that either 
$q_1 = yz$ or $q_1 = z^2$. In the first case we eliminate the mixed term in $q_2$ to get $q_2 = d'y^2 + f'z^2$. In the other case,
we eliminate the $z^2$ term in $q_2$ and get $q_2= d'y^2 + e'yz$. At the point we reverse the roles get $q_2 = yz$ and get
$q_1 = a'y^2 + c'z^2$.

\medskip

We carry out a last change of variables and set $y=h_1$ and $z=h_2$. In the
first case we have
\[ I = ( x^2, xy, xz, y^2, z^2),\]
while in the second case we have
\[ I = ( x^2, xy, xz, yz, ay^2 +cz^2).\] 
In both cases, by the proof of  Proposition~\ref{Prop2.1}, the ideal $I$ is not syzygetic because $I^2 \neq \m^4$.

\medskip

 By faithful flatness, we may pass to the algebraic closure of $K$ and, by rescaling, we may assume $a=b=1$.
 Then these ideals have reduction number $2$
since  $(x^2, y^2, z^2)$ and $(x^2, yz, y^2+z^2)$ are respective minimal reductions.
\end{proof}

To complete the picture we should compute $\e_1(I)$ for submaximally generated ideal $I$.

\begin{Proposition} \label{Prop2.3}
Let $\RR = K[x,y,z]$ with $K$ infinite field,
 and $I$ an ideal of codimension $3$ minimally generated by $5$ quadrics.
\begin{itemize}
\item[{\rm (i)}] $\e_1(I)=4$ and thus $\RR[It]$ has the condition $R_1$ of Serre;
\item[{\rm (ii)}] If $I$ is not Gorenstein, then $\depth  \gr_{I}(\RR) = 2$.
\end{itemize}
\end{Proposition}

\begin{proof}

\noindent (i) Let $\m$ be a maximal ideal containing $I$.
 Since $\m^2$ is the integral closure of $I$, we have $\e_{1}(I) \leq \e_{1}(\m^2) =4$. Moreover, we make use of  the multiplicity of the Sally module of $I$ (\cite[Theorem 2.11]{icbook}):
\[ 0< s_0(I) = \e_1(I) - \e_0(I) + \lambda(\RR/I) = \e_{1}(I) - 8 +5. \]
This implies that $\e_{1}(I)=4$. In particular, the equality $\e_1(I) = \e_1(\m^2)$ implies the condition $R_1$ for $\RR[It]$.

\medskip

\noindent (ii) Suppose that $I$ is not Gorenstein. By Proposition~\ref{Prop2.2}, the ideal
$I$ is either $ ( x^2, xy, xz, y^2, z^2)$ or $ ( x^2, xy, xz, yz, y^2 +z^2)$ up to change of variables and by passing to the algebraic closure of $K$. We claim that, in both cases, $\lambda(I^{2}/JI)=1$ for a minimal reduction $J$ of $I$. If $I= ( x^2, xy, xz, y^2, z^2)$, then we may choose $J=(x^2,y^2,z^2)$. Notice that $x^2yz=xy \cdot xz$ is the only generator of $I^2$ not in $JI$.
Thus, $I^2/JI \simeq R/(JI:(x^2yz))$. On the other hand, we have
                \[ x \cdot x^2yz = z \cdot x^3y \in JI, \quad 
                 y \cdot x^2yz = z \cdot x^2y^2 \in JI, \quad
                 z \cdot x^2yz = y \cdot x^2z^2 \in JI. \]
Therefore, $\lambda(I^2/JI)=1$. Now suppose that $I=(x^2,xy,xz,yz,y^2+z^2)$. Then we may choose 
    $J=(x^2,yz,y^2+z^2)$ so that $x^2y^2=(xy)^2$ is the only generator of $I^2$ not in $JI$.
Repeating the same argument as above, we find
\[ x \cdot x^2y^2 = y \cdot x^3y \in JI, \quad 
                 y \cdot x^2y^2 = x \cdot xy^3 \in JI, \quad
                 z \cdot x^2y^2 = y \cdot x^2yz \in JI,\]
which means that $\lambda(I^2/JI)=1$. Hence by \cite[Theorem 4.7]{HucMar97}, 
 $\depth  \gr_{I}(\RR) \geq 2$. Since $I\neq \m^2$,
by Serre's normality criterion we cannot have $\depth  \gr_{I}(\RR)=3$ as this would be equivalent to $\RR[It]$ being
Cohen--Macaulay and then $\RR[It] = \RR[\m^2t]$.

\end{proof}

In the next corollary we compute the Castelnuovo regularity of the Rees algebras of the ideals above. We are going 
to be guided by the techniques of \cite{Trung98}. In view of the fact that 
$\reg(\gr_I(\RR)) = \reg(\RR[It])$ (\cite[Proposition 4.1]{JU96}, \cite[Lemma 4.8]{Ooishi82}), we will find more convenient to deal with the associated graded ring. 

\medskip

\begin{Corollary} \label{regdim3}
For all these ideals the Castelnuovo regularity of the Rees algebras $\RR[It]$ is $2$, in particular these algebras
have relation type at most $3$ $($precisely $3$ in the Gorenstein case$)$.
\end{Corollary}

\begin{proof}
We separate the two cases.
Suppose that $I$ is not syzygetic. Then $\depth \gr_I(\RR) = 2$. According to \cite[Theorem 6.4]{Trung98} (see also \cite[Corollary 2.2]{Marley93}), $\reg(\RR[It]) = \red(I) = 2$.

\medskip

If $I$ is syzygetic, set $\AA = \gr_I(\RR)$ and $\BB =  \gr_{\m^2}(\RR)$. 
We are going to make use of the fact that $\m^2/I \simeq K $ and $I^n = \m^{2n}$ for all $n \geq 2$.
  Consider the exact sequence of finitely generated graded modules $\AA$-modules
\[ 0 \rar K \lar \AA \lar \BB \lar K[-1] \rar 0,\]
where $K$ and $K[-1]$ are copies of $K$ concentrated in degrees $0$ and $1$ (resp.), which we break up
into two short exact sequences:
\[ 0 \rar K \lar \AA \lar L  \rar 0,\]
\[ 0 \rar L  \lar \BB \lar K[-1] \rar 0.\]
To compute the Castelnuovo regularity of $\AA$ we apply the local cohomology functors 
$\H^i(\cdot) = \H_{\AA_{+}}^i( \cdot)$ and obtain the isomorphisms (recall that $\BB$ is
Cohen-Macaulay):
\[ \H^{0}(\AA)= K, \quad \H^{1}(\AA)=K[-1], \quad \H^{2}(\AA)=0, \quad \H^{3}(\AA)=\H^{3}(\BB).\]
Denoting by $a_i(\cdot)$ the function that reads the top degree of these cohomology modules, in the
formula
\[ \reg(\AA) = \max\{ a_i(\AA)+ i \mid 0\leq i \leq 3\}, \]
therefore
\[ \reg(\AA)= \max\{0, 2, 0, a_3(\BB)+3\}.\]
But $a_3(\BB)+3 = \reg(\BB)= \red(\m^2) = 1$, since $\BB$ is 
Cohen-Macaulay. Thus $\reg(\AA)=2$.
An alternative proof would be by the use of \cite[Proposition 4]{Ooishi82}.

\end{proof}

\begin{Remark} \label{remregdim3}{\rm 
A similar statement and proof will apply to the special fiber $\mathcal{F}(I) = \RR[It]\otimes \RR/\m$ in the
case of Gorenstein ideals.
We consider the commutative diagram
 \[
\diagram
0 \rto & \m\RR[It] \rto\dto  & \m\RR[\m^2t] 
\rto\dto  & 0 \rto\dto & 0 \\
0 \rto & \RR[It] \rto\dto  & \RR[\m^2t] 
\rto\dto  & K[-1] \rto\dto & 0 \\
0 \rto & \mathcal{F}(I)\rto  & \mathcal{F}(\m^2) 
\rto  & K[-1] \rto & 0. 
\enddiagram
\]
Since $\m^2/I \simeq K$, $I^n= \m^{2n}$ for all $n\geq 2$ and $\m^3 = \m I$, the two top rows are exact. Thus
the bottom row is also exact  and therefore both the depth and regularity of $\mathcal{F}(I)$ can
be read off it, in particular $\depth \mathcal{F}(I) = 1$.

\medskip

We can also calculate the depth of the special fiber for
all the other ideals generated by $5$ quadrics. We may assume that $K$ is algebraically closed. 
Our discussion
showed that they all belong to the orbits of $(x^2, xy, xz, y^2, z^2)$ and $(x^2, xy, xz, yz, y^2+z^2)$.
Treating each by {\em Macaulay2}, they both have Cohen--Macaulay special fibers. In particular, they do not satisfy the condition $R_1$.

}\end{Remark}

\subsubsection*{Ideals generated by $4$ quadrics}

Now we complete our study of the Rees algebras of ideals generated by quadrics in $3$-space
by considering  ideals $I$ generated by $4$ quadrics.

\begin{Proposition}\label{4quadrics}
Let $I=(J,a)$ be an ideal minimally generated by $4$ quadrics of codimension $3$ in $\RR=K[x,y,z]$, where $J$ is a minimal reduction. Then 

\begin{enumerate}
\item[{\rm (i)}] $I$ has reduction number $1$ or $3$, and $I$ is never syzygetic. If $\red(I) = 1$,
$\RR[It]$ is Cohen-Macaulay and does not satisfy the condition $R_1$.

\item[{\rm (ii)}] If $\red(I) = 3$, $\depth \RR[It] = 3$ and satisfies the condition $R_1$.

\end{enumerate}
\end{Proposition}

\begin{proof}
 We already observed that $\lambda(\RR/I) = 6$ and $\e_1(I) \leq 4$. We recall
that $\e_1(I) = 4$ means that $\RR[It]$ has the condition $R_1$ of Serre.
By \cite[Proposition 3.3]{syl2} $\red(I) = 1, 3$, where $\red(I)=3$
corresponds to the condition $R_1$. Moreover, 
an $\m$-primary almost complete intersection, in any dimension,  is never syzygetic
according to the formula of    \cite[Proposition 3.7]{syl2}. 
 
\medskip

Note that $\lambda(I/J) =2$. Let $u=\lambda(I^2/JI)$ and $v=\lambda(I^3/JI^2)$.
By \cite[Proposition 3.12]{acm}, $\lambda(I^2/JI) = 1$ or $0$, depending on whether $J:I =  \m$ or not.
If $u=v=1$, by
Huckaba's theorem (\cite[Theorem 2.1]{Huc96},
\cite[Theorem 4.7(b)]{HucMar97}) 
 we would have that $\depth \ G = 2$.  

 If $u=1$, we have $I^2 = (ab)+JI$, $a,b\in I$. We claim that $I^3 = (a^2b)+ JI^2$. Indeed, for
$u,v,w\in I$ we have $uv = rab +s, \quad s \in JI$. Therefore
$ uvw = rabw + ws $, and if we write $bw = pab + q$, where  $q\in I$, we have
\[ uvw = rp a^2b  + A, \quad A = ws + raq\in JI^2.\]  
 We now have the data to apply Huckaba's theorem (\cite[Theorem 2.1]{Huc96}): $\e_1(I) = 4$, $\e_0(I) = 8$, 
$\lambda(\RR/I) = 6$, $\lambda(I^2/JI) = \lambda(I^3/JI^2)= 1$,
\[ 4-8+6 = 1+1,\]
and thus $\depth \gr_I(\RR) \geq 2$, but as we cannot have $\RR[It]$ to be Cohen-Macaulay as otherwise,
with the condition $R_1$ we would have $\RR[It] = \RR[\m^2t]$.

\end{proof}

\section{Finiteness of Gorenstein ideals} Our next goal is an analysis of the conjectural fact that the number
of Gorenstein ideals of $\RR=K[x_{1}, \ldots, x_{d}]$ generated by $\nu(\m^{2})-1$  quadrics is finite, up to change of
variables. The correct number may depend 
on $K$. 

\medskip

\begin{Theorem} \label{abc} Let $\RR=K[x_1, \ldots, x_d]$ with $d \geq 3$ and $\m=(x_{1}, \ldots, x_{d})$. Let $K=\mathbb{R}$ and let
 $I$ be an $\RR$--ideal of codimension $d$ minimally generated by $\nu(\m^{2})-1$ quadrics.
If $I$ is a Gorenstein ideal then up to change of variables, $I$ is one of the following $d$ ideals
\[ I = (x_{i} x_{j}, x_{1}^2\pm x_{k}^2 \mid 1 \leq i < j \leq d, \; k=2, \ldots, d).\]
The converse holds for all fields $K$.
\end{Theorem}

\begin{proof} Let $I$ be one of these ideals and set
$\AA = \RR/I$. Fix $h\in \m^2$ a generator of the socle of $\AA$. Let $V= \AA_1$ the $d$-dimensional
vector space generated by the linear forms of $\AA$. 
On $V$ define the scalar product $\circ$ that makes 
 $\{x_{1}, \ldots, x_{d} \}$  an orthonormal basis. Finally define a quadratic form on $V$ by the rule
 \[ u, v\in V\quad uv = C(u,v)h, \quad C(u,v) \in \mathbb{R}.\]
 $C(\cdot, \cdot)$ is a symmetric bilinear form: let  $\BB$ be the corresponding symmetric matrix,
 \[ C(u,v) = u\circ \BB(v).\]
Since the socle of $\AA$ contains no linear form, $\BB$ is nonsingular of eigenvalues $\{a_{1}, \ldots, a_{d}\}$.
 By the spectral
Theorem, $\BB = \PP\DD \PP^{-1}$ , where $\PP$ is orthogonal and $\DD$ is the  diagonal matrix
$[a_{1}: a_{2}: \cdots: a_{d}]$. 
The change of variables
\[ \{X_{1}, X_{2}, \ldots, X_{d}\} = \PP \{x_{1}, x_{2}, \ldots, x_{d} \}\]
then give that the quadrics
\[\{ X_{i}X_{j}, a_{2}X_{1}^2-a_{1}X_{2}^2, \ldots, a_{d}X_{1}^2-a_{1}X_{d}^2\} \subset I,\]
and therefore generate the ideal $I$ since they are linearly independent. By rescaling and taking appropriate square roots we conclude that $I$ is one of $(x_{i} x_{j}, x_{1}^2\pm x_{k}^2 \mid 1 \leq i < j \leq d, \; k=2, \ldots, d)$.

\medskip

 For the converse, let $I=(x_{i} x_{j}, x_{1}^2\pm x_{k}^2 \mid 1 \leq i < j \leq d, \; k=2, \ldots, d)$.  Note that $\m^{3}=\m I$, and therefore the Hilbert function of $\AA$ is $(1,d, 1)$.
Suppose that $I$ is not a Gorenstein ideal. Then there exists a linear form $h$ in the socle of $\AA$, that is
 $h \cdot \m \subset I$. Let $h=\sum_{j=1}^{d} a_{j}x_{j}$, where $a_{i} \in K$. Then, for each $i=1, \ldots, d$, we have $x_{i}\sum_{j=1}^{d} a_{j}x_{j} \in I$. Since $x_{i}x_{j} \in I$ for all $1 \leq i <j \leq d$, we get $a_{i}x_{i}^{2} \in I$, which means that $a_{i}=0$ for all $i$. Hence $h=0$, which is a contradiction.
\end{proof}

\begin{Example}{\rm Let $\AA$ be a $3\times 3$ matrix with linear entries in $\RR = \mathbb{R}[x_1,x_2,x_3,x_4]$.
If the ideal $I = I_2(\AA)$ has codimension $4$ then after change of variables it can be generated by
$6$ monomials and $3$ binomials.
}\end{Example} 

\medskip

\begin{Corollary}
Let $\RR=\mathbb{R}[x_1, \ldots, x_d]$ with $d \geq 3$ and $\m=(x_{1}, \ldots, x_{d})$. 
Let $I$ be a Gorenstein $\RR$--ideal of codimension $d$ minimally generated by $\nu(\m^{2})-1$ quadrics.
Then
\begin{itemize}
\item[{\rm (i)}] $I^2 = \m^4$.
\item[{\rm (ii)}] $\RR[It]$ has the condition $R_1$ and $\depth \gr_{I}(R) = 0$. 
\item[{\rm (iii)}] $\e_{1}(I)=(d-1)2^{d-2}$.
\item[{\rm (iv)}] If $d=3$, then $\red(I)=2$. If $d \geq 4$, then $d=2 \cdot \red(I)$ or $d= 2 \cdot \red(I)+1$.
\item[{\rm (v)}] If $d \geq 4$, then $I$ is not syzygetic. 
\end{itemize}
\end{Corollary}

\begin{proof}(i) It is enough to prove that $x_{i}^{4} \in I^{2}$ for all $i=1, \ldots, d$. This follows from
\[ x_{i}^{4} = (x_{i}^{2} \pm x_{j}^{2})(x_{i}^{2} \pm x_{k}^{2}) \pm (x_{i}x_{k})^{2} \pm (x_{i}x_{j})^{2} \pm (x_{j}x_{k})^{2} \in I^{2},\]
where $i, j, k$ are all distinct.

\medskip

\noindent (ii) It follows from the similar argument given in the proof of Proposition~\ref{Prop2.1}.

\medskip

\noindent (iii) Since $\m^{3}=\m I$ and $\m^{4}=I^{2}$, we have $I^{n} = \m^{2n}$ for all $n \geq 2$. In particular, we obtain an equality of respective Hilbert polynomials:
\[ \e_{0}(I) {{n+d}\choose{d}} - \e_{1}(I) {{n+d-1}\choose{d-1}} + \cdots = { {2n+1+d}\choose{d}}.\]
By comparing the coefficients of $n^{d-1}$ of both sides, we get
\[ \frac{\e_{0}(I) d(d+1)}{2 d !} - \frac{ \e_{1}(I)}{(d-1)!} = \frac{ (d^{2}+3d)2^{d-2}}{d!}. \]
Now we use $\e_{0}(I)=2^{d}$ and solve this equation for $\e_{1}(I)$, which proves the assertion.

\medskip

\noindent (iv) Since $I^{n} = \m^{2n}$ for all $n \geq 2$, for all $n \geq 2$, we get
\[  \m^{2(n+1)} = I^{n+1}= J I^{n} = J \m^{2n}.\] 
Since both reduction numbers of $I$ and $\m^{2}$ are greater than $1$, 
it implies that  the reduction numbers of $I$ and $\m^{2}$ are equal.

\medskip

\noindent (v)
Suppose that $I$ is sygygetic. Then by
\[ {{d+3}\choose{4}} = \nu(\m^{4})= \nu(I^{2}) = {{\nu(I)+1}\choose{2}} = \frac{ (d+2)(d+1)d(d-1)}{8}.\]
By solving this for $d$, we get $d=3$.
\end{proof}

\begin{Corollary}\label{regGordimd} Let $I$ be a Gorenstein ideal as above. Then for $d\geq 4$, 
$\reg(\RR[It]) = \red(\m^2)$, in particular the relation type of
$\RR[It]$ is at most $\red(\m^2) + 1$.
\end{Corollary}

\begin{proof} It is the same calculation as Corollary~\ref{regdim3}, except that now we use that
$\red(\m^2)\geq 2$.
\end{proof}

As in Remark~\ref{remregdim3}, we also have the depth of $\mathcal{F}(I)$ is $1$.

\subsubsection*{Submaximally  generated ideals of quadrics}

Let $\RR = K[x_1, \ldots, x_d]$, $\m=(x_{1}, \ldots, x_{d})$,  and let $I$ be an ideal of codimension $d$ minimally generated by $\nu(\m^{2})-1$ quadrics.   It is obvious that such ideals can be generated by
the same number of binomials. [Just pick a monomial $A\notin I$ and use that $I + (A)= \m^2$.] What 
we would like is a statement in the spirit of Theorem~\ref{abc}, 
\[ I = 
\mbox{\rm (monomials) $+$ ($r$ binomials), $r< d$}.\]

Making use of Theorem~\ref{abc}
we can now state a structure theorem for almost maximally generated ideals generated by quadrics.

\begin{Theorem} \label{abc2} Let $\RR = \mathbb{R}[x_1,\ldots, x_d]$, $\m=(x_{1}, \ldots, x_{d})$ and $I$ an ideal of
codimension $d$ minimally generated by $\nu(\m^{2})-1$ quadrics, let $r+1$ be the Cohen--Macaulay type of $I$.
Then up to a change of variables
\[ I  = (x_1, \ldots, x_r)\m + I_{d-r},\]
where $I_{d-r} $ is a Gorenstein ideal of $\mathbb{R}[x_{r+1}, \ldots, x_d]$. With additional  
  changes  $I$ can 
be written as
\begin{eqnarray} \label{abc2f} I = K_{r,d} + K_{d-r} + (x_1^2, \ldots, x_r^2) + (x_{r+1}^2\pm x_{r+2}^2, \ldots, x_{r+1}^2\pm x_d^2),
\end{eqnarray}
where $K_{r,d} + (x_1^2, \ldots, x_r^2) = (x_1,\ldots, x_r)(x_1, \ldots, x_d)$ and $K_{d-r}$ is the edge ideal
on the complete graph of vertex set $\{x_{r+1}, \ldots, x_d\}$. In particular
\[ I = (\mbox{\rm monomials}) + (d-r-1 \ \mbox{\rm binomials}).\]
Conversely, any ideal as in {\rm $($\ref{abc2f}$)$}, over any field $K$, has Cohen-Macaulay type $r+1$.
\end{Theorem}

\begin{proof}
We give first some observations about the Hilbert function of $\AA=\RR/I$.
If $I$ a Gorenstein ideal of codimension $d$,  the Hilbert function of $\AA$ is simply $(1, d, 1)$.

\medskip

 In general the Hilbert function of $\AA$ has the form $(1, d, 1, \ldots)$ with a string
of $s$ $1$'s of length $1\leq s\leq d-1$
  because once a 
component has dimension $1$ those that follow also have dimension $1$ or $0$. If $\AA$ is not Gorenstein and $s>2$,
 its
socle has a generator of degree $s$, but no generator of degrees  $[2, s-1]$, since for instance $\AA_s = \AA_1\cdot 
\AA_{s-1}\neq 0$. This implies that besides the generator in degree $s$ the other generators are in degree $1$.
In particular this means that there is a linear form $h$ such that $h\cdot \m\subset I$.

\medskip 
 
 We could actually do better:  Let $S$ be the socle of $\AA$ and $S_1 = S\cap \AA_1$. Then $S_1$
is a vector space of dimension $r$, where $ r= \mbox{\rm type}(I)-1$.  Assume that $r\neq 0$, that is $I$ is 
not Gorenstein. 

\medskip

 Let $\{x_1, \ldots, x_r\}$ (after relabeling) a set of linear forms that generate $S_1$. Denote
$\RR' = \RR/(x_1, \ldots, x_r)$, let $I'$ be the image of $I$ in $\RR'$. Then
 \[ I = (x_1, \ldots, x_r)\m + I'\RR,\]
 and 
\[ I\cap (x_1, \ldots, x_r)\RR = (x_1, \ldots, x_r)\m, \] which means that we have embeddings
\[ I/(x_1, \ldots, x_r)\m \hookrightarrow \RR', \quad \AA/(S_1) \simeq \RR'/I' = \AA'.\]

\medskip

 The ideal $I'$ is a submaximal ideal of $\RR'$ 
 and the Hilbert function of $\AA'$ has the
form $(1, d-r, 1, \ldots)$. Furthermore the strands $(\ldots)$ of the Hilbert functions of $\AA$ and $\AA'$
are identical. Morevoer $I'$ is a Gorenstein ideal: Otherwise we would have 
a linear form $g\in \RR'$ with $g \m'\subset I'$, and since we already have $g\cdot (x_1,\ldots, x_r)\subset I$,
we would have $g\m \subset I$, which is a contradiction to the definition of $S_1$.  Note that in particular
the strand $(\ldots)$ consists of $0$'s.

\medskip

Finally the monomials are assembled  and we apply Theorem~\ref{abc} to $I'$. 
For the final assertion, we note that $(x_1, \ldots, x_r)$ lies in the socle of  $\AA$ and thus $\mbox{\rm type}(I) \geq r+1$. For the converse one uses induction on
sequences of the kind $0\rar I/x_1\m \rar \RR/(x_1)=\RR' \rar \RR'/I' \rar 0$
and $0\rar (x_1)/x_1\m \rar \RR/I \rar \RR'/I' \rar 0$
 to calculate types.  
\end{proof}

A more abstract, but less detailed structural description of Gorenstein ideals that holds for all fields is
the following polarized version of Theorem~\ref{abc}.

\begin{Theorem} \label{abc3} Let $\RR =K[x_1,\ldots, x_d]$ and $I$ a Gorenstein ideal of codimension $d$
 minimally generated by $\nu(\m^{2})-1$ 
 quadrics. Then there is a set of linearly independent forms $\{y_1, \ldots,y_d\}$ such that
\[ I = (x_iy_j, 1\leq i,j \leq d, i\neq j, x_1y_1-x_2y_2, \ldots, x_1y_1-x_dy_d).\]
Conversely, any ideal of this form is Gorenstein.
\end{Theorem}

\begin{proof}
We use the notation of the proof of Theorem~\ref{abc}. Consider the bilinear form
\[ (u,v) \mapsto uv= C(u,v)h, \quad C(u,v) \in K.\]
$C(\cdot, \cdot)$ is a non degenerate bilinear form  since $\AA$ is a Gorenstein algebra. Fix the basis
$\{x_1, \ldots, x_d\}$ and pick its dual basis 
of linear forms $\{y_1, \ldots, y_d\}$
with respect to $C$. This means that $C(x_i,y_j) = \delta_{i,j}$, so that we have
\[ J= (x_iy_j, 1\leq i,j \leq d, i\neq j, x_1y_1-x_2y_2, \ldots, x_1y_1-x_dy_d)\subset I. \]
Note that $J $ has codimension $d$ and \[J + (x_1y_1) = (x_1,\ldots, x_d)(y_1, \ldots, y_d) = \m^2.\]
In other words $J $ is submaximally generated and thus $J=I$. 

\medskip 
The converse uses the same argument as the proof of Theorem~\ref{abc}. 
\end{proof}

\begin{Corollary} \label{abc3Cor}
If I is a Gorenstein ideal of $K[x_1, \ldots, x_d]$ of codimension $d$ submaximally generated by
quadrics, then $I^2 = \m^4$. 
\end{Corollary}

\begin{proof} $\m^2 = (x_1, \ldots, x_d)(y_1, \ldots, y_d)$, so
\[ \m^4 = (x_iy_jx_ky_{\ell}, 1\leq i,j,k,\ell\leq d).\] Let us check the two types of critical monomials:
\[ x_1^2y_1^2 = (x_1y_1-x_2y_2)(x_1y_1-x_3y_3) + (x_1y_3)(x_3y_1) + (x_1y_2)(x_2y_1) - (x_2y_3)(x_3y_2)\in I^2,\]
\[ x_1y_1x_1y_2 = x_1y_2(x_1y_1 - x_3y_3) + x_1y_3x_3y_2,\] 
and similar ones.
\end{proof}


\begin{thebibliography}{99}


\bibitem{BH}{W. Bruns and J. Herzog, {\em Cohen--Macaulay Rings}, Cambridge University Press, 
Cambridge, 1993.}


\bibitem{Macaulay2}{D. Grayson and M. Stillman,
  {Macaulay 2}, a software system for research in algebraic
  geometry, 2006.
  Available at {\tt http://www.math.uiuc.edu/Macaulay2/}.}


\bibitem{HSV83}{J. Herzog, A. Simis and W. V. Vasconcelos,
Koszul homology and blowing-up rings, in Proc.
Trento Conf. in Comm. Algebra, Lect. Notes in Pure and Applied Math. {\bf 84}, Marcel Dekker,
New York, 1983, 79--169.}


\bibitem{syl2}{J. Hong, A. Simis and W. V. Vasconcelos, On the equations
of almost complete intersections,  {\em Bull. Braz. Math.  Soc.} {\bf 43} (2012), 171-199.}



\bibitem{acm}{J. Hong, A. Simis and W. V. Vasconcelos, Extremal Rees algebras,  {\em J. Commutative Algebra} {\bf 5} (2013), 231-267.}


\bibitem{Huc96}{S. Huckaba, A $d$-dimensional extension of a lemma of
Huneke's and formulas for the Hilbert coefficients, {\em Proc. Amer.
Math. Soc.} {\bf 124} (1996), 1393-1401.}

\bibitem{HucMar97}{S. Huckaba and T. Marley, Hilbert coefficients
 and the depths of associated graded
rings, {\em J. London Math. Soc.}  {\bf 56} (1997), 64--76.}


\bibitem{Hu0}{C. Huneke, On the associated graded ring of an ideal,
{\em Illinois J. Math.} {\bf 26} (1982), 121-137.}

\bibitem{JU96}{M. Johnson and B. Ulrich, Artin--Nagata properties and Cohen--Macaulay associated graded rings, 
{\em Compositio Mathematica} {\bf 103} (1996), 7--29.}

\bibitem{Marley93}{T. Marley, The reduction number of an ideal and the local Cohomology of the associated graded ring, Proc. Amer. Math. Soc. {\bf 117}(1993), 335--341.}


\bibitem{Ooishi82}{A. Ooishi, Castelnuovo's regularity of graded rings and modules,
 {\em Hiroshima Math. J.} {\bf 12} (1982),
627--644.}




\bibitem{Planas98}{ F. Planas--Vilanova, On the module of effective relations of a standard algebra,
 {\em Math. Proc. Camb. Phil. Soc.} {\bf 124} (1998), 215--229.}


\bibitem{Trung98}{N. V. Trung, The Castelnuovo regularity of the Rees algebra and associated graded ring,
{\em Trans. Amer. Math. Soc.} {\bf 350} (1998), 2813--2832.}


\bibitem{icbook}{W. V. Vasconcelos, 
{\em Integral Closure}, Springer Monographs in Mathematics, New York, 2005. }

\end{thebibliography}
\end{document}